\documentclass[preprint,12pt]{elsarticle}

\usepackage{amssymb, amsthm,amsfonts,amsbsy}

\usepackage{graphicx}
\newtheorem{theorem}{Theorem}
\newtheorem{lemma}{Lemma}
\newdefinition{definition}{Definition}
\newdefinition{remark}{Remark}
\newdefinition{question}{Question}

\begin{document}
\begin{frontmatter}
\title{Selective but not Ramsey}

\author{Timothy Trujillo}
\address{University of Denver, Department of Mathematics, 2360 S Gaylord St, Denver, CO 80208, USA}
\ead{Timothy.Trujillo@du.edu}

\begin{abstract}
We give a partial answer to the following question of Dobrinen: For a given topological Ramsey space $\mathcal{R}$, are the notions of selective for $\mathcal{R}$ and Ramsey for $\mathcal{R}$ equivalent? Every topological Ramsey space $\mathcal{R}$ has an associated notion of Ramsey ultrafilter for $\mathcal{R}$ and selective ultrafilter for $\mathcal{R}$ (see \cite{MijaresSelective}).   If $\mathcal{R}$ is taken to be the Ellentuck space then the two concepts reduce to the familiar notions of Ramsey and selective ultrafilters on $\omega$; so by a well-known result of Kunen the two are equivalent. We give the first example of  an ultrafilter on a topological Ramsey space that is selective but not Ramsey for the space, and in fact a countable collection of such examples.

For each positive integer $n$ we show that for the topological Ramsey space $\mathcal{R}_{n}$ from \cite{Ramsey-Class2}, the notions of selective for $\mathcal{R}_{n}$ and Ramsey for $\mathcal{R}_{n}$ are not equivalent. In particular, we prove that forcing with a closely related space using almost-reduction, adjoins an ultrafilter that is selective but not Ramsey for $\mathcal{R}_{n}$.
Moreover, we introduce a notion of finite product among members of the family $\{\mathcal{R}_{n}: n<\omega\}$. We show that forcing with closely related product spaces using almost-reduction, adjoins ultrafilters that are selective but not Ramsey for these product topological Ramsey spaces. 
\end{abstract}


\end{frontmatter}

\section{Introduction}
\label{section1}
This paper is concerned with giving examples of topological Ramsey spaces $\mathcal{R}$ and ultrafilters that are selective for $\mathcal{R}$ but not Ramsey for $\mathcal{R}$. The first result of topological Ramsey theory was the infinite dimensional extension of the Ramsey theorem known as the Ellentuck theorem (see \cite{Ellentuck}). Ellentuck proved this theorem in order to give a proof of Silver's theorem stating that analytic sets have the Ramsey property. In order to state the Ellentuck theorem it is necessary to introduce the Ellentuck space.

We denote the infinite subsets of $\omega$ by $[\omega]^{\omega}$ and the finite subsets of $\omega$ by $[\omega]^{<\omega}$. If $B\in[\omega]^{\omega}$ and $\{b_{0},b_{1},b_{2}, \dots\}$ is its increasing enumeration, then for each $i<\omega$, we let $r_{i}(B)$ denote the set $\{b_{0},b_{1},b_{2},\dots, b_{i-1}\}$ and call it the $i^{th}$ approximation of $B$. The \emph{Ellentuck space} is the set $[\omega]^{\omega}$ of all infinite subsets of $\omega$ with the topology generated by the basic open sets, 
\begin{equation}
[a,B]= \{ A \in [\omega]^{\omega}:   A \subseteq B \ \& \ (\exists i)  r_{i}(B) =a \}
\end{equation} where $a\in[\omega]^{<\omega}$ and $B\in [\omega]^{\omega}$.

 Recall that a subset of a topological space is {\it nowhere dense} if its closure has empty interior and {\it meager} if it is the countable union of nowhere dense sets. A subset ${\mathcal X}$ of a topological space has the {\it Baire property} if and only if ${\mathcal X} = {\mathcal O} \cap {\mathcal M}$ for some open set ${\mathcal O}$ and some meager set ${\mathcal M}$.

 A subset ${\mathcal X}$ of the Ellentuck space is {\it Ramsey} if for every $\emptyset \not = [a,A]$, there is a $B\in[a,A]$ such that $[a,B] \subseteq {\mathcal X}$ or $[a,B] \cap {\mathcal X} = \emptyset$. The next theorem is the infinite-dimensional version of the Ramsey theorem.
\begin{theorem}[Ellentuck Theorem, \cite{Ellentuck}]
\label{ET}
Every subset of the Ellentuck space with the Baire property is Ramsey. 
\end{theorem}

Topological Ramsey spaces are spaces that have enough structure in common with the Ellentuck space that an abstract version of the Ellentuck theorem can be stated and proved. The Ellentuck space leads naturally to the notion of a selective ultrafilter on $\omega$.

\begin{definition}
Let $\mathcal{U}$ be a nonprincipal ultrafilter on $\omega$. If $i<\omega$ and $A$ is an infinite subset of $\omega$, \emph{i.e.} in the Ellentuck space, then we let
\begin{equation} A/i= A\setminus r_{i}(A).
\end{equation}
 $\mathcal{U}$ is \emph{selective}, if for each decreasing sequence $A_{0} \supseteq A_{1} \supseteq \dots $ of members of ${\mathcal U}$ there exists $X=\{x_{0}, x_{1}, \dots \}\in{\mathcal U}$ enumerated in increasing order such that for all $i<\omega$, 
\begin{equation}
A/i \subseteq A_{i}.
\end{equation}
\end{definition}  The next theorem, due to Kunen, characterizes selective ultrafilters as those which are minimal with respect to the Rudin-Keisler ordering.  
\begin{theorem}[\cite{Booth}]
Let $\mathcal{U}$ be an ultrafilter on $\omega$. The following conditions are equivalent:
\begin{enumerate}
\item $\mathcal{U}$ is selective.
\item For each partition of the two-element subsets of $\omega$ into two parts, there is a set $X\in\mathcal{U}$ all of whose two-element subsets lie in one part of the partition.
\item Every function on $\omega$ is constant or one-to-one on some set in $\mathcal{U}$.
\end{enumerate}
\end{theorem}
An ultrafilter that satisfies the second item is called a \emph{Ramsey} ultrafilter on $\omega$. Generalizations of the previous theorem have been studied in many contexts. For example, the notions of selective coideal (see \cite{Mathias}) and semiselective coideals (see \cite{Farah}) have been shown to also satisfy similar Ramsey properties. In \cite{MijaresSelective}, Mijares generalizes the notion of selective ultrafilter on $\omega$ to a notion of selective ultrafilter on an arbitrary topological Ramsey space $\mathcal{R}$. Mijares also generalizes the notion of Ramsey ultrafilter on $\omega$ to a notion of Ramsey ultrafilter for $\mathcal{R}$ and shows that if an ultrafilter is Ramsey for $\mathcal{R}$ then it is also selective for $\mathcal{R}$. If one takes $\mathcal{R}$ to be the Ellentuck space then the two generalizations reduce to the concepts of selective and Ramsey ultrafilter. The theorem of Kunen above shows that the notions of selective for the Ellentuck space and Ramsey for the Ellentuck space are equivalent. This leads to the following question asked by Dobrinen about the generalizations from selective and Ramsey to arbitrary topological Ramsey spaces.
\begin{question}
 For a given topological Ramsey space $\mathcal{R}$, are the notions of selective for $\mathcal{R}$ and Ramsey for $\mathcal{R}$ equivalent?
\end{question} 
Ramsey for $\mathcal{R}$ ultrafilters have also been studied by Dobrinen and Todorcevic in \cite{Ramsey-Class} and \cite{Ramsey-Class2}. Motivated by Tukey classification problems, the authors develop a hierarchy of topological Ramsey spaces $\mathcal{R}_{\alpha}$, $\alpha\le \omega_{1}$.  Associated to each space $\mathcal{R}_{\alpha}$ is an ultrafilter $\mathcal{U}_{\alpha}$, which is Ramsey for $\mathcal{R}_{\alpha}$. The space $\mathcal{R}_{0}$ is taken to be the Ellentuck space; therefore, Ramsey for $\mathcal{R}_{0}$ is equivalent to selective for $\mathcal{R}_{0}$. We show that for each positive integer $n$, there is a triple $(\mathcal{R}^{\star}_{n},\le, r)$ such that forcing with the space using almost-reduction, adjoins an ultrafilter that is selective for $\mathcal{R}_{n}$ but not Ramsey for $\mathcal{R}_{n}$. 

In Section $\ref{section2}$, we introduce the concept of a topological Ramsey space. The main theorem of this section is the abstract Ellentuck theorem due to Carlson and Simpson. We follow the presentation of Todorcevic in \cite{RamseySpaces} and introduce four axioms which can be used to state the abstract Ellentuck theorem. 

In Section $\ref{section3}$, we give the general setting for the main results in this article. For each positive integer $k$ and each tree $T$ on $\omega^{k}$, satisfying some conditions, we associate a triple $(\mathcal{R}(T), \le, r)$.  We also introduce the generalization of Ramsey and selective for ultrafilters on the maximal nodes of $T$ that we use in this article. (In the following we denote the maximal nodes of a tree by $[T]$.)

Section $\ref{section4}$ consists of the archetype example for the methods we apply in this article. In this section, we introduce the tree $T_{1}$ and the triple $(\mathcal{R}(T_{1}), \le, r)$ from \cite{Ramsey-Class} which we denote by $(\mathcal{R}_{1}, \le, r)$. We then construct a closely associated tree $T_{1}^{\star}$ and prove that the associated triple $\mathcal{R}(T^{\star}_{1})$ forms a topological Ramsey space. Then we show that forcing with $\mathcal{R}^{\star}_{1}$ using almost-reduction, adjoins a selective but not Ramsey for ${\mathcal R}_{1}$ ultrafilter on $[T_{1}]$.

In section $\ref{section5}$, for each positive integer $n$, we introduce the tree $T_{n}$ and the triple $(\mathcal{R}(T_{n}), \le, r)$ from \cite{Ramsey-Class2} which we denote by $(\mathcal{R}_{n}, \le, r)$. We then construct a closely associated tree $T_{n}^{\star}$ and space $\mathcal{R}(T^{\star}_{n})$. We show that forcing with $\mathcal{R}^{\star}_{n}$ using almost-reduction, adjoins a selective but not Ramsey for ${\mathcal R}_{n}$ ultrafilter on $[T_{n}]$.

In section $\ref{section6}$, we consider finite sequences $\left < S_{i} : i\le n \right>$ where each $S_{i}$ is one of the trees $T_{j}$ for some $j<\omega$. We introduce the product $\bigotimes_{i=0}^{n}\mathcal{R}(S_{i})$ from \cite{GenRamsey-Class}. Then we construct a closely associated tree $\bigotimes_{i=0}^{n} S^{\star}_{i}$ and space $\bigotimes_{i=0}^{n}\mathcal{R}^{\star}(S_{i})$.  We prove that forcing with  $\bigotimes_{i=0}^{n}\mathcal{R}^{\star}(S_{i})$ using almost-reduction, adjoins a selective but not Ramsey for $\bigotimes_{i=0}^{n}\mathcal{R}(S_{i})$ ultrafilter on $[\bigotimes_{i=0}^{n} S_{i}]$.

In section $\ref{section7}$, we discuss why the methods used in this article fail for some topological Ramsey spaces defined from similar types of trees. We conclude with some questions about the generalizations of Ramsey and selective ultrafilters to the spaces where our methods fail.

In this article we use the methods of forcing but all of our constructions can be carried out using  CH or MA. We work with $\sigma$-closed partial orders and all of the constructions only require $2^{\aleph_{0}}$ conditions to be met. For example, assuming CH we can guarantee the conditions hold at successor stages and use $\sigma$-closure at limit stages.  

The author would like to express his deepest gratitude to Natasha Dobrinen for valuable comments and suggestions that helped make this article and its proofs more readable. 
\section{Background}
 \label{section2}
A topological Ramsey space $\mathcal{R}$, by definition, is a space that satisfies an abstract version of the Ellentuck theorem. In order to state an abstract version of the Ramsey property for $\mathcal{R}$ it is necessary to have an abstract notion of the partial order ``$\subseteq$" and an abstraction notion the restriction map ``$r$". To this end, we consider triples $({\mathcal R}, \le, r)$ where $\mathcal{R}$ is a nonempty set, $\le$ is a quasi-ordering on ${\mathcal R}$ and $r:\mathcal{R} \times \omega \rightarrow  \mathcal{AR}$. For each such triple we can define an abstract notion of Ramsey and endow $\mathcal{R}$ with a topology similar to the Ellentuck space.

\begin{definition}
Let $({\mathcal R}, \le, r)$ be a triple such that $\mathcal{R}$ is a nonempty set, $\le$ is a quasi-ordering on ${\mathcal R}$ and $r:\mathcal{R} \times \omega \rightarrow  \mathcal{AR}$ is surjective. For each $a \in \mathcal{AR}$ and each $B\in \mathcal{ R}$, let 
\begin{equation} [a, B] = \{ A \in {\mathcal R}: A \le B \ \& \ (\exists n) r_{n}(A) = a \}.
\end{equation} The \emph{Ellentuck topology} on ${\mathcal R}$ is the topology generated by the sets $[a,B]$ where $a\in \mathcal{AR}$ and $B\in \mathcal{R}$.

A subset ${\mathcal X}$ of ${\mathcal R}$ is \emph{Ramsey} if for every $\emptyset \not= [a,A],$ there is a $B \in [a,A]$ such that $[a,B] \subseteq {\mathcal X}$ or $[a,B]\cap {\mathcal X} = \emptyset.$
A subset ${\mathcal X}$ of ${\mathcal R}$ is \emph{Ramsey null} if for every $\emptyset \not= [a,A],$ there is a $B \in [a,A]$ such that $[a,B]\cap {\mathcal X} = \emptyset.$

A triple $({\mathcal R}, \le, r )$ with its Ellentuck topology is a \emph{topological Ramsey space} if every subset of $\mathcal{R}$ with the Baire property is Ramsey and if every meager subset of ${\mathcal R}$ is Ramsey null.
\end{definition}

We follow the presentation of the abstract Ellentuck theorem given by Todorcevic in \cite{RamseySpaces}, rather than the earlier reference \cite{CarlsonSimpson}. In particular, we introduce four axioms about triples $({\mathcal R}, \le, r )$ sufficient for proving an abstract version of the Ellentuck theorem. The first axiom we consider tells us that $\mathcal{R}$ is collection of infinite sequences of objects and $\mathcal{AR}$ is collection of finite sequences approximating these infinite sequences. 

\noindent \emph{{\bf A.1} For each $A,B \in {\mathcal R}$, \begin{enumerate}
		\item[(a)] $r_{0}(A) = \emptyset$.
		\item[(b)] $A\not = B$ implies $r_{i}(A) \not = r_{i}(B)$ for some $i$.
		\item[(c)] $r_{i}(A) = r_{j}(B)$ implies $i=j$ and $r_{k}(A) = r_{k}(B)$ for all $k<i$.
	\end{enumerate}
}
On the basis of this axiom, $\mathcal{R}$ can be identified with a subset of $\mathcal{AR}^{\omega}$ by associating $A\in\mathcal{R}$ with the sequence $(r_{i}(A))_{i<\omega}$. Similarly, $a\in\mathcal{AR}$ can be identified with $(r_{i}(A))_{i<j}$ where $j$ is the unique natural number such that $a=r_{j}(A)$ for some $A\in\mathcal{R}$. For each $a\in \mathcal{AR}$, let $|a|$ equal the natural numbers $i$ for which $a=r_{i}(a)$. For $a,b \in \mathcal{AR}$, $a \sqsubseteq b$ if and only if $a=r_{i}(b)$ for some $i \le |b|$. $a \sqsubset b$ if and only if $a=r_{i}(b)$ for some $i<|b|$.

\noindent\emph{{\bf A.2} There is a quasi-ordering $\le_{\mathrm{fin}}$ on $\mathcal{AR}$ such that
	\begin{enumerate}
		\item[(a)] $\{a \in \mathcal{AR} : a \le_{\mathrm{fin}} b\}$ is finite for all $b \in {\mathcal AR}$,
		\item[(b)] $ A \le B$ iff $(\forall i)(\exists j) \ r_{i}(A) \le_{\mathrm{fin}} r_{j}(B),$
		\item[(c)] $\forall a,b,c \in {\mathcal AR} [ a \sqsubseteq b \wedge b \le_{\mathrm{fin}} c \rightarrow \exists d \sqsubseteq c \ a \le_{\mathrm{fin}} d ].$
	\end{enumerate}}
For $a\in\mathcal{AR}$ and $B\in\mathcal{R}$ $\mathrm{depth}_{B}(a)$ is the least $i$, if it exists, such that $ a \le_{\mathrm{fin}} r_{i}(B)$. If such an $i$ does not exist, then we write $\mathrm{depth}_{B}(a) = \infty$. If depth$_{B}(a)=i< \infty,$ then $[\mathrm{depth}_{B}(a),B]$ denotes $[r_{i}(a), B]$.  

\noindent\emph{{\bf A.3} For each $A, B\in\mathcal{ R}$ and each $a\in\mathcal{AR}$,
	\begin{enumerate}
		\item[(a)] If $\mathrm{depth}_{B}(a) < \infty$ then $[a,A]\not = \emptyset$ for all $A \in [\mathrm{depth}_{B}(a),B]$.
		\item[(b)] $A \le B$ and $[a, A]\not = \emptyset$ imply that there is an $A' \in [ \mathrm{depth}_{B}(a),B]$ such that $\emptyset \not = [a, A'] \subseteq [a, A]$.
	\end{enumerate}}
If $n>|a|$, then $r_{n}[a,A]$ denotes the collection of all $b\in {\mathcal AR}_{n}$ such that $ a \sqsubset b$ and $b \le_{\mathrm{fin}} A$.

\noindent \emph{{\bf A.4} For each $B\in{\mathcal R}$ and each $a\in\mathcal {AR}$, if depth$_{B}(a) < \infty$ and $\mathcal{ O} \subseteq \mathcal{AR}_{|a|+1}$, then there is $A \in [\mathrm{depth}_{B}(a), B]$ such that
\begin{equation} r_{|a|+1}[a,A] \subseteq {\mathcal O}\mbox{ or } r_{|a|+1}[a, A] \subseteq {\mathcal O}^{c}.\end{equation}
}
The next result, using a slightly different set of axioms, is a theorem of Carlson and Simpson in \cite{CarlsonSimpson}. The version using {\bf A.1}-{\bf A.4} can be found as Theorem 5.4 in \cite{RamseySpaces}.
\begin{theorem}[Abstract Ellentuck theorem, \cite{CarlsonSimpson}] If $(\mathcal{R}, \le , r)$ is a closed subspace of $\mathcal{AR}^{\omega}$ and satisfies {\bf A.1, A.2, A.3} and {\bf A.4} then $(\mathcal{R}, \le , r)$ forms a topological Ramsey space.
\end{theorem}

\section{General setting}
\label{section3}
In this section, in order to avoid repeating similar definitions, we introduce a framework for constructing triples from trees.  For each set $X$, $X^{<\omega}$ denotes the collection of all finite sequences of elements of $X$. For each finite sequence $s$, we let $|s|$ denote the length of $s$. For each $i\le|s|$, $\pi_{i}(s)$ denotes the sequence of the first $i$ elements of $s$ and $s_{i}$ denotes the $i^{th}$ element of the sequence. For each pair of sequences $s$ and $t$, we say that $s$ is an \emph{initial segment} of $t$ and write $s\sqsubseteq t$ if there exists $i\le|t|$ such that $s=\pi_{i}(t)$.

The \emph{closure} of $T \subseteq X^{<\omega}$ (denoted by $cl(T)$) is the set of all initial segments of elements of $T$. A subset $T$ of $X^{<\omega}$ is a \emph{tree} on $X$, if $cl(T)=T$. A \emph{maximal node} of $T$, is a sequence $s$ in $T$ such that for each $t\in T$, $s \sqsubseteq t \Rightarrow s=t$. The \emph{body} of $T$ (denoted by $[T]$) is the set of all maximal nodes of $T$. The \emph{height} of $T$ is the smallest ordinal greater than or equal to the length of each element of $T$. 

Let $k$ be a positive integer. The \emph{lexicographical order} of $(\omega^{k})^{<\omega}$ is defined as follows: $s$ is lexicographically less than $t$ if and only if $s\sqsubseteq t$ or $|s|=|t|$ and the least $i$ on which $s$ and $t$ disagree, $s_{i} \le t_{i}$ where $\le$ is taken to be the product order on $\omega^{k}$. If $S$ and $T$ are trees on $\omega^{k}$, then $S$ is \emph{isomorphic} to $T$, if there exists a bijection $h:S \rightarrow T$ which preserves the lexicographic ordering. A \emph{subtree} of $T$ is a tree $S$ such that $S\subseteq T$. Given two trees $S$ and $T$ on $\omega^{k}$, we let ${ T \choose S}$ denote the set of all subtrees of $T$ that are isomorphic to $S$. If $S$ is a subtree of $T$ then we will write $S\le T.$

For each positive integer $k$ and trees $S$, $T$ and $U$ on $\omega^{k}$, the partition relation
\begin{equation}
T \rightarrow(S)^{U}
\end{equation}
means that for each partition of ${T\choose U}$ into two parts there exists $V\in{ T \choose S}$ such that ${V \choose U}$ lies in one part of the partition.

\begin{definition}[($\mathcal{R}(T),\le, r$)]
Suppose that $k$ is a positive integer. Let $T$ be a tree on $\omega^{k}$ such that for all $s,t\in[T]$, $|s|=|t|$ and $\pi_{0}''[T]=\{ \left<(n,\dots, n)\right>\in \omega^{k}: n<\omega\}$.  Let $\mathcal{R}(T)$ denote the set of all subtrees of $T$ isomorphic to $T$, \emph{i.e.} ${T\choose T}$. For each $S\in \mathcal{R}(T)$ we let $\{\left< (k^{S}_{i}, \dots, k^{S}_{i}) \right> : i<\omega\}$ denote the lexicographically increasing enumeration of $\pi_{0}''[S]$. For each $i<\omega$, let
\begin{equation}
S(i) = cl(\{ s\in [S]: \pi_{0}(s) =\left< (k^{S}_{i}, \dots, k^{S}_{i}) \right>\})
\end{equation}
Let $\mathcal{AR}(T) = \bigcup_{i<\omega} \{ r_{i}(S) : S\in\mathcal{R}(T)\}$ and define $r: \omega\times \mathcal{R}(T) \rightarrow \mathcal{AR}(T)$ by letting $r(i, S) = \bigcup_{j<i} S(j)$. For $S,S'\in\mathcal{R}_{1}$, $S\le S'$ if and only if $S$ is subtree of $S'$. For $S,S'\in\mathcal{R}(T)$ \emph{almost-reduction} is defined as follows: $S\le^{*} S'$ if and only if there exists $i<\omega$ such that $S\setminus r_{i}(S) \subseteq S'.$
\end{definition}
Next following Dobrinen and Todorcevic in \cite{Ramsey-Class} and \cite{Ramsey-Class2} we introduce a generalization of the notion of Ramsey and selective for triples built from trees.

\begin{definition} \label{RamseyDef}  Let $k$ be a positive integer and $T$ be a tree on $\omega^{k}$. Suppose that $(\mathcal{R}(T),\le,r)$ satisfies $\bf{A.1}$-$\bf{A.4}$ and forms a topological Ramsey space.  Let $\mathcal{U}$ be an ultrafilter on $[T]$. 
\begin{enumerate}
\item We say that $\mathcal{U}$ is \emph{generated by $\mathcal{G} \subseteq \mathcal{R}(T)$}, if $\{[S] : S \in\mathcal{G}\}$ is cofinal in $( \mathcal{U},\supseteq)$.
\item An ultrafilter $\mathcal{U}$ generated by $\mathcal{G} \subseteq \mathcal{R}(T)$ is \emph{selective for $\mathcal{R}(T)$} if and only if for each decreasing sequence $S_{0}\ge S_{1} \ge S_{2} \ge \dots $ of elements of $\mathcal{G}$, there exists another $S\in\mathcal{G}$ such that for all $i<\omega$, $S \setminus r_{i}(S)\subseteq S_{i}.$
\item An ultrafilter $\mathcal{U}$ generated by $\mathcal{G} \subseteq \mathcal{R}(T)$ is \emph{Ramsey for $\mathcal{R}(T)$} if and only if for each $i<\omega$ and each partition of ${ T \choose r_{i}(T)}$ into two parts there exists $S\in\mathcal{G}$ such that ${S \choose r_{i}(T)}$ lies in one part of the partition. 
\end{enumerate}
\end{definition}

The next result addresses the existence of Ramsey ultrafilters for $\mathcal{R}(T)$. We omit its proof as it follows by applying Lemma 3.3 of Mijares in \cite{MijaresSelective} to topological Ramsey spaces of the form $\mathcal{R}(T)$.
\begin{theorem}[\cite{MijaresSelective}] Let $k$ be a positive integer and $T$ be a tree on $\omega^{k}$. Suppose that $(\mathcal{R}(T),\le,r)$ satisfies $\bf{A.1}$-$\bf{A.4}$ and forms a topological Ramsey space.  Forcing with $(\mathcal{R}(T), \le^{*})$ adjoins no new elements of $(\mathcal{AR}(T))^{\omega}$, and if $\mathcal{G}$ is a $(\mathcal{R}(T), \le^{*})$-generic filter over some ground model $V$, then $\mathcal{G}$ generates a Ramsey for $\mathcal{R}(T)$ ultrafilter in $V[\mathcal{G}]$.
\end{theorem}

In the next section, we give an example of a tree $T$ defined by Dobrinen and Todorcevic in \cite{Ramsey-Class} where the notions of Ramsey for $\mathcal{R}(T)$ and selective for $\mathcal{R}(T)$ ultrafilters on $[T]$ are not equivalent.

\section{Selective but not Ramsey for $\mathcal{R}_{1}$}
\label{section4}
We begin this section with the definition of the triple $(\mathcal{R}_{1}, \le, r)$. This space was first defined by Dobrinen and Todorcevic in \cite{Ramsey-Class}. The construction of $\mathcal{R}_{1}$ in \cite{Ramsey-Class} was inspired and motivated by the work of Laflamme in \cite{Laflamme} which uses forcing to adjoin a weakly-Ramsey ultrafilter satisfying complete combinatorics over $\mathrm{HOD}(\mathbb{R})$.

The purpose of this section is to introduce a closely related triple $(\mathcal{R}^{\star}_{1}, \le, r)$ and show that forcing with $\mathcal{R}^{\star}_{1}$ using almost-reduction, adjoins an ultrafilter that is selective but not Ramsey for $\mathcal{R}_{1}$.
\begin{definition}[$(\mathcal{R}_{1}, \le, r)$, \cite{Ramsey-Class}]
For each $i<\omega$, let
\begin{equation}
T_{1}(i) =\left  \{ \left< \ \right>, \left < i \right >, \left< i, j \right >:  \frac{i(i+1)}{2}\le j <\frac{(i+1)(i+2)}{2} \right\}
\end{equation}
Let $T_{1} = \bigcup_{i<\omega} T_{1}(i)$ and $(\mathcal{R}_{1},\le,r)$ denote the triple $(\mathcal{R}(T_{1}), \le, r)$. Figure $\ref{graphT1}$ includes a graph of the tree $T_{1}$.

\end{definition}
The next result is Theorem 3.9 of Dobrinen and Todorcevic in \cite{Ramsey-Class}.
\begin{theorem}[\cite{Ramsey-Class}]
$(\mathcal{R}_{1}, \le, r)$ satisfies $\bf{A.1}$-$\bf{A.4}$ and  forms a topological Ramsey space.
\end{theorem}

The next result is a consequence of Theorem 3.5 of Mijares in \cite{MijaresGalvin} applied to the topological Ramsey space $\mathcal{R}_{1}$.
\begin{lemma}[\cite{MijaresGalvin}]
\label{finiteRamseyR1}
For each pair of positive integer $k$ and $n$ with $k\le n$ there exists $m<\omega$ such that
\begin{equation}
r_{m}(T_{1}) \rightarrow (r_{k}( T_{1}))^{r_{k}(T_{1})}.
\end{equation}
\end{lemma}
Next we define the topological Ramsey space $\mathcal{R}_{1}^{\star}$. The space is constructed from a modified version of the tree $T_{1}$. The modified tree $T_{1}^{\star}$ has height $3$ and for each $i<\omega$, the maximal nodes of $T^{\star}_{1}(i)$ are in one-to-one correspondence with the maximal nodes of $T_{1}(i)$. These two properties of $T_{1}^{\star}$ are used below to show that forcing with $\mathcal{R}_{1}^{\star}$ using almost-reduction adjoins an ultrafilter that is selective but not Ramsey for $\mathcal{R}_{1}$. The one-to-one correspondence is used to show that the adjoined ultrafilter is selective for  $\mathcal{R}_{1}$ and the extra level of $T_{1}^{\star}$ is used to show that the adjoined ultrafilter fails to be Ramsey for $\mathcal{R}_{1}$.
\begin{definition}[$(\mathcal{R}^{\star}_{1}, \le , r)$]  For each $i<\omega$, let
\begin{equation}
T_{1}^{\star}(i)= cl(\{ \left<i,j,k \right> : k\le i \ \& \ \left<j,k \right>\in T_{1}\}).
\end{equation}
Let $T_{1}^{\star}= \bigcup_{i<\omega} T_{1}^{\star}(i)$ and $(\mathcal{R}^{\star}_{1},\le,r)$ denote the triple $(\mathcal{R}(T^{\star}_{1}), \le, r)$. Figure $\ref{graphT1}$ is a graph of the tree $T^{\star}_{1}$.
\end{definition}
\begin{figure}[h!]
\includegraphics{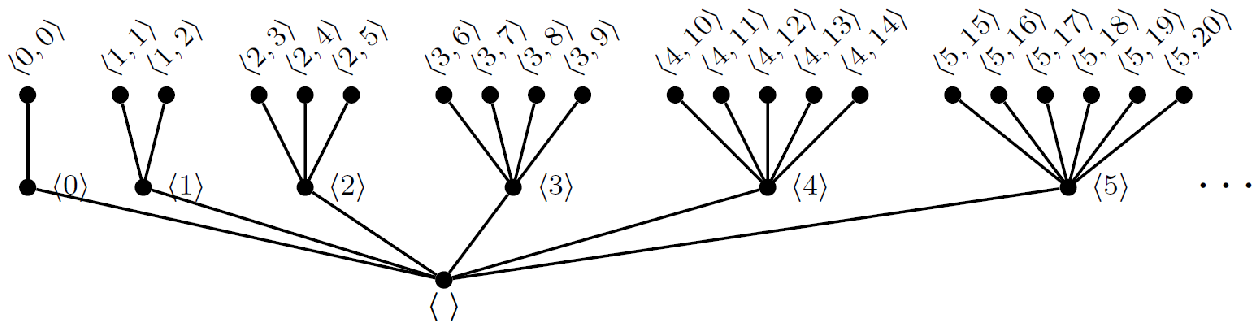}
\includegraphics{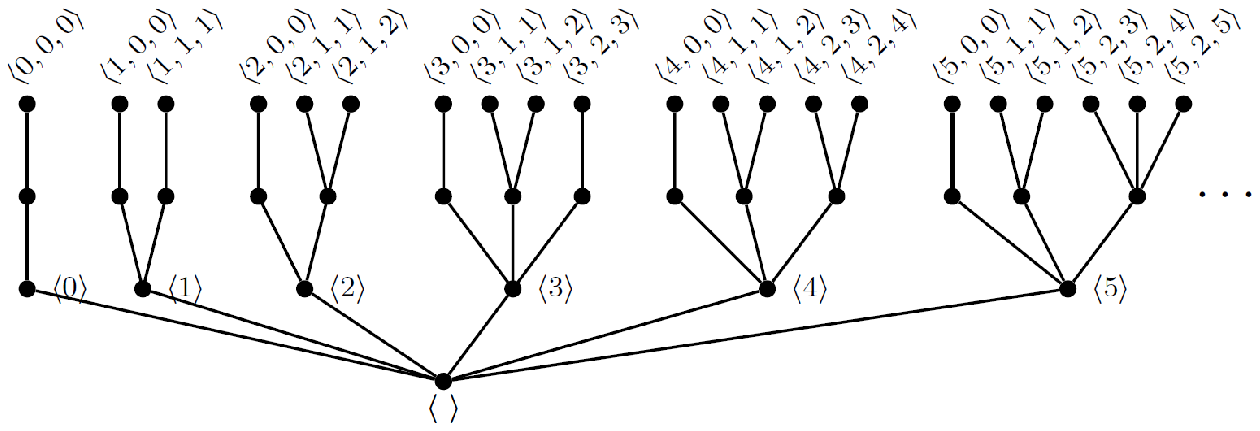}
\caption{Graph of $T_{1}$ and $T_{1}^{\star}$}
\label{graphT1}
\end{figure}
The next two partition properties are need to show that $\mathcal{R}^{\star}_{1}$ satisfies axiom $\bf{A.4}$.
\begin{lemma}
\label{BaseCase}
For each pair of positive integers $k$ and $n$ with $k \le n$, there exists $m<\omega$ such that
\begin{equation}
T_{1}^{\star}(m) \rightarrow ( T_{1}^{\star}(n))^{T_{1}^{\star}(k)}.
\end{equation}
\end{lemma}
\begin{proof} Let $k$ and $n$ be positive integers. For each $i<\omega$, let $i'$ be the smallest natural number such that $T_{1}^{\star}(i)$ is isomorphic to a subtree of $cl(\{ \left< i'\right >^{\frown}s : s\in r_{i'}(T_{1})\})$. By Lemma $\ref{finiteRamseyR1}$ there exists $m<\omega$ such  that $r_{m}(T_{1}) \rightarrow (r_{n'}( T_{1}))^{T_{1}(k')}.$  Suppose that $\{\Pi_{0}, \Pi_{1}\}$ is a partition of ${ T_{1}^{\star}(m) \choose T_{1}^{\star}(k)}$. For each $j<2$, let $\Pi_{j}' = \{ S\in  {r_{m}(T_{1})\choose r_{k'}(T_{1})} : cl(\{\left<m' \right>^{\frown}s: s\in [\hat{S}]\}) \in \Pi_{j}\}$ where $[\hat{S}]$ consists of the lexicographically first $k$ elements of $[S]$. $\{\Pi'_{0}, \Pi'_{1}\}$ forms a partition of ${r_{m}(T_{1}) \choose r_{k'}(T_{1})}$. Hence, there exists $j<2$ and $S\in { r_{m}(T_{1}) \choose  r_{n'}(T_{1}) }$ such that ${S \choose r_{k'}(T_{1})}\subseteq \Pi_{j}'$. If we let $S':=cl(\{\left<m' \right>^{\frown}s: s\in [\hat{S}]\}) \in \Pi_{j}$ where $[\hat{S}]$ consists of the lexicographically first $n$ elements of $[S]$, then $S'\in{T^{\star}_{1}(m) \choose T_{1}^{\star}(n)}$, and ${S' \choose T_{1}^{\star}(k)}\subseteq \Pi_{j}$. Therefore the lemma holds.
\end{proof}

\begin{lemma} For each positive integer $k$,
\label{pigeonhole}
\begin{equation}
T_{1}^{\star} \rightarrow (T^{\star}_{1})^{T_{1}^{\star}(k)}.
\end{equation}
\end{lemma}
\begin{proof} Let $k$ be a positive integer. Lemma $\ref{BaseCase}$ shows that there exists a strictly increasing sequence $(m_{i})_{i<\omega}$ such that for each $i<\omega$, $T^{\star}_{1}(m_{i}) \rightarrow (T^{\star}_{1}(i))^{T_{1}^{\star}(k)}$. Let $\{\Pi_{0}, \Pi_{1}\}$ be a partition of ${T_{1}^{\star}\choose T_{1}^{\star}(k)}$ and $(S_{0}, S_{1}, \dots)$ be a sequence of trees such that for each $i<\omega$, $S_{i} \in { T_{1}^{\star} \choose T_{1}^{\star}(i)}$ and ${S_{i} \choose T_{1}^{\star}(k)}$ is contained in one piece of the partition $\{\Pi_{0}, \Pi_{1}\}$. By the pigeonhole principle there exists $j<2$ and a strictly increasing sequence $(i_{0}, i_{1}, \dots)$  such that for all $l<\omega$, $S_{i_{l}} \in { T_{1}^{\star}(m_{i_{l}})\choose T_{1}^{\star}(i)}$ and ${S_{i_{l}} \choose T_{1}^{\star}(k)}\subseteq \Pi_{j}$. Let $S = \bigcup_{l<\omega} S_{i_{l}}$. If $S'$ is any element of ${S \choose T_{1}^{\star}}$ then ${ S' \choose T_{1}^{\star}(k)} \subseteq \Pi_{j}$. Therefore the lemma holds.
\end{proof}
\begin{theorem}
\label{RamseySpace}
$(\mathcal{R}^{\star}_{1}, \le, r)$ satisfies $\bf{A.1}$-$\bf{A.4}$ and forms a topological Ramsey space.
\end{theorem}
\begin{proof} By the abstract Ellentuck theorem it is enough to show that $(\mathcal{R}^{\star}_{1}, \le, r)$ satisfies $\bf{A.1}$-$\bf{A.4}$ and forms  a closed subspace of $(\mathcal{AR}_{1})^{\omega}$. The proof that $\mathcal{R}^{\star}_{1}$ is a closed subspace of $(\mathcal{AR}_{1})^{\omega}$ and satisfies axioms $\bf{A.1}$-$\bf{A.3}$ follows by trivial modifications to the proofs of the same facts for the space $\mathcal{R}_{1}$ in \cite{Ramsey-Class}. For this reason, we omit the proof that $\mathcal{R}^{\star}_{1}$ forms a closed subspace of $(\mathcal{AR}_{1})^{\omega}$ and satisfies axioms $\bf{A.1}-\bf{A.3}$. By definition of $\mathcal{R}_{1}^{\star}$, $\bf{A.4}$ is equivalent to Lemma $\ref{pigeonhole}$. Hence $\bf{A.4}$ holds for $\mathcal{R}_{1}^{\star}$. By the abstract Ellentuck theorem, $\mathcal{R}^{\star}_{1}$ forms a topological Ramsey space.
\end{proof}

\begin{lemma}
\label{sigma-closed}
For each sequence  $S_{0} \reflectbox{$\le^{*}$} S_{1}  \reflectbox{$\le^{*}$} S_{2} \dots$ of elements of $\mathcal{R}^{\star}_{1}$ there exists $S\in\mathcal{R}_{1}^{\star}$ such that for all $i<\omega$, $S\setminus r_{i}(S) \subseteq S_{i}$.
\end{lemma}
\begin{proof}
Let $S_{0}\reflectbox{$\le^{*}$} S_{1}\reflectbox{$\le^{*}$} S_{2} \dots$ be a decreasing sequence in $\mathcal{R}_{1}^{\star}$. Hence, there is a strictly increasing sequence $(k_{i})_{i<\omega}$ of natural numbers such that for all $i<\omega$ and all $j<i$, $S_{i+1}\setminus r_{k_{i}}(S_{i+1}) \subseteq S_{j}$. For each $i<\omega$, let $S(i)$ be an element of ${S_{i+1}\setminus r(k_{i},S_{i+1}) \choose T^{\star}(i)}.$ Let $S= \bigcup_{i<\omega} S(i)$. Then $S\in \mathcal{R}_{1}^{\star}$ and for all $i<\omega,$ $S\setminus r_{i}(S) \subseteq S_{i}$.
\end{proof}
Next we define maps $\gamma$ and $\Gamma$ that will be used to transfer an ultrafilter on $[T_{1}^{\star}]$ generated by a subset of $\mathcal{R}^{\star}_{1}$ to an ultrafilter on $[T_{1}]$ generated by a subset of $\mathcal{R}_{1}$.
\begin{definition}
Let $\{t_{0},t_{1},t_{2}, \dots\}$ and $\{s_{0},s_{1},s_{2}, \dots\}$ be the lexicographically increasing enumeration of $[T_{1}]$ and $[T_{1}^{\star}]$, respectively. Let $\gamma:[T_{1}^{\star}]\rightarrow [T_{1}]$ such that for all $i<\omega$,
\begin{equation}
 \gamma(s_{i})=t_{i}.
\end{equation}
 Let $\Gamma: \mathcal{R}_{1}^{\star}\rightarrow \mathcal{R}_{1}$ be the map given by 
\begin{equation}
\Gamma(S) = cl( \gamma''[S]).
\end{equation}
\end{definition}
\begin{remark}
$\gamma$ is bijective and $\Gamma$ is injective but not surjective. 
\end{remark}

\begin{theorem}
\label{ultrafilter}
 $(\mathcal{R}_{1}^{\star}, \le^{*})$ is $\sigma$-closed, and if $\mathcal{G}$ is a generic filter for $(\mathcal{R}_{1}^{\star}, \le^{*})$ over some ground model $V$, then $\Gamma''\mathcal{G}$ generates an ultrafilter on $[T_{1}]$ that is selective for $\mathcal{R}_{1}$ but not Ramsey for $\mathcal{R}_{1}$ in $V[\mathcal{G}]$.
\end{theorem}
\begin{proof}

By Lemma $\ref{sigma-closed}$, $(\mathcal{R}^{\star}_{1},\le^{*})$ is a $\sigma$-closed partial order. Suppose that $\mathcal{G}$ is a generic filter for $(\mathcal{R}^{\star}_{1},\le^{*})$ and $X \subseteq [T_{1}]$. By the previous remarks, $X$ is in the ground model $V$. Since $[T_{1}]$ is in bijective correspondence with ${T^{\star}_{1} \choose T^{\star}_{1}(0)}$, Lemma $\ref{pigeonhole}$ shows that for each $T\in\mathcal{R}_{1}^{\star}$ there exists $S\in{ T \choose T^{\star}_{1}}$ such that either $\gamma''[S] \subseteq X$ or $\gamma''[S] \cap X = \emptyset$. Hence,
\begin{equation}
\Delta_{X} = \{ S \in \mathcal{R}^{\star}_{1}: [\Gamma(S)] \subseteq X \mbox{ or } [\Gamma(S)] \cap X = \emptyset\}
\end{equation} 
is dense in $(\mathcal{R}^{\star}_{1},\le^{*})$. Since $\mathcal{G}$ is generic, for each $X\subseteq[T_{1}]$, $\mathcal{G} \cap \Delta_{X} \not= \emptyset$. In particular, for each $X\subseteq [T_{1}]$ there exists $S\in \mathcal{G}$ such that $[\gamma(S)]\subseteq X$ or $[\gamma(S)]\cap X=\emptyset$. Therefore $\Gamma''\mathcal{G}$ generates an ultrafilter on $[T_{1}]$.  Let $\mathcal{V}_{1}$ denote the ultrafilter on $[T_{1}]$ generated by $\Gamma''\mathcal{G}$.

 Let $(S_{0}, S_{1}, \dots)$ be a sequence of elements of $\mathcal{R}_{1}$ such that $\Gamma(S_{0}) \ge \Gamma(S_{1})  \ge  \Gamma(S_{2}) \ge \dots$ is a decreasing sequence. Since $(\mathcal{R}_{1 },\le^{*})$ is $\sigma$-closed,
\begin{equation}
\Delta_{(S_{0}, S_{1}, \dots)}= \bigcup_{i<\omega} \{ S\in \mathcal{R}_{1}^{\star}: S \cap S_{i} = \{\left < \ \right> \} \}\cup \bigcap_{i<\omega} \{ S\in \mathcal{R}_{1}^{\star}: S\setminus r_{i}(S) \subseteq S_{i}\}
\end{equation}
is in the ground model $V$. 

Next we show that $\Delta_{(S_{0}, S_{1}, \dots)}$ is dense in $(\mathcal{R}^{\star}_{1},\le^{*})$. To this end, suppose that $T\in\mathcal{R}_{1}^{\star}$. Since $\mathcal{R}_{1}^{\star}$ forms a topological Ramsey space, either there exists $S'\le T$  and $i<\omega$ such that $S_{i} \cap S' = \{ \left< \ \right>\}$ or there exists a sequence $(S'_{0}, S'_{1}, \dots)$ in $[\emptyset, T]$ such that for each $i<\omega$, $S'_{i} \le S_{i}$. In the first case, $S' \le^{*} T$ and $S' \in \Delta_{(S_{0}, S_{1}, \dots)}$. In the second case, Lemma \ref{sigma-closed} shows that there exists $S\le T$ such that $ S \setminus r_{i}(S) \subseteq S'_{i} \subseteq S_{i}$. In particular, $S\in \Delta_{(S_{0}, S_{1}, \dots)}$. In both cases, there exists $S\le T$ such that $S\in \Delta_{(S_{0}, S_{1}, \dots)}$. Therefore $\Delta_{(S_{0}, S_{1}, \dots)}$ is dense  in $(\mathcal{R}^{\star}_{1},\le^{*})$.

Suppose that the sequence $\Gamma(S_{0}) \ge \Gamma(S_{1})  \ge  \Gamma(S_{2}) \dots$ consists of members $\Gamma''\mathcal{G}$. Then $\mathcal{G} \cap  \Delta_{(S_{0}, S_{1}, \dots)} \not= \emptyset$ shows that there exists $S\in \mathcal{G}$ such that for all $i<\omega$, $\Gamma(S) \setminus r_{i}(\Gamma(S)) \subseteq \Gamma(S_{i})$. Therefore $\mathcal{V}_{1}$ is a selective for $\mathcal{R}_{1}$ ultrafilter on $[T_{1}]$.

Next we construct a partition of ${ T_{1} \choose r_{2}(T_{1})}$ and show that the partition witnesses that $\mathcal{V}_{1}$ is not Ramsey for $\mathcal{R}_{1}$. For each $s \in{ T_{1} \choose r_{2}(T_{1})}$, we let $s'$ and $s''$ denote the two lexicographically smallest elements of $[s\setminus  r_{1}(s)]$. Notice that for each $s\in{ T_{1} \choose r_{2}(T_{1})}$ the length of the longest common initial segment of $\gamma^{-1}(s')$ and $\gamma^{-1}(s'')$ is either $1$ or $2$. For each $j<2$, let $\Pi_{j}$ denote the set of all $s\in{ T_{1} \choose r_{2}(T_{1})}$ such that length of the longest common initial segment of $\gamma^{-1}(s')$ and $\gamma^{-1}(s'')$ is $1-j$. For each $S\in \Gamma''\mathcal{G}$, ${ S\choose r_{2}(T_{1})}$ is neither a subset of $\Pi_{0}$ nor $\Pi_{1}$. Therefore $\mathcal{V}_{1}$ is not a Ramsey for $\mathcal{R}_{1}$ ultrafilter on $[T_{1}]$.
\end{proof}

\section{Selective but not Ramsey for $\mathcal{R}_{n}$}
\label{section5}

In this section, we investigate the triples $(\mathcal{R}_{n}, \le, r)$ for $n<\omega$. These spaces were first defined by Dobrinen and Todorcevic in \cite{Ramsey-Class2}. The construction of $\mathcal{R}_{n}$ in \cite{Ramsey-Class} was motivated by the work of Laflamme in \cite{Laflamme} which uses forcing to adjoin a $(n+1)$-Ramsey ultrafilter having exactly $n$ Rudin-Keisler predecessors, a linearly ordered chain of p-points.
The purpose of this section is to introduce a closely related triple $(\mathcal{R}^{\star}_{n}, \le, r)$ and show that forcing with $\mathcal{R}^{\star}_{n}$ using almost-reduction, adjoins an ultrafilter that is selective but not Ramsey for $\mathcal{R}_{n}$.

\begin{definition}[$(\mathcal{R}_{n}, \le, r)$] Assume $n$ is a positive integer and $T_{1}, T_{2}, \dots, T_{n}$ have been defined. For each $i<\omega$, let
\begin{equation}
T_{n+1}(i) =\left  \{ \left< \ \right>, \left < i \right >, \left< i\right >^{\frown}s: s\in T_{n}(j) \ \& \  \frac{i(i+1)}{2}\le j <\frac{(i+1)(i+2)}{2} \right\}.
\end{equation}
Let $T_{n+1} = \bigcup_{i<\omega} T_{n+1}(i)$ and $(\mathcal{R}_{n+1},\le,r)$ denote the triple $(\mathcal{R}(T_{n+1}), \le, r)$. Figure $\ref{graphT2}$ includes a graph of the tree $T_{2}$.
\end{definition}
\begin{figure}[h!]
\includegraphics{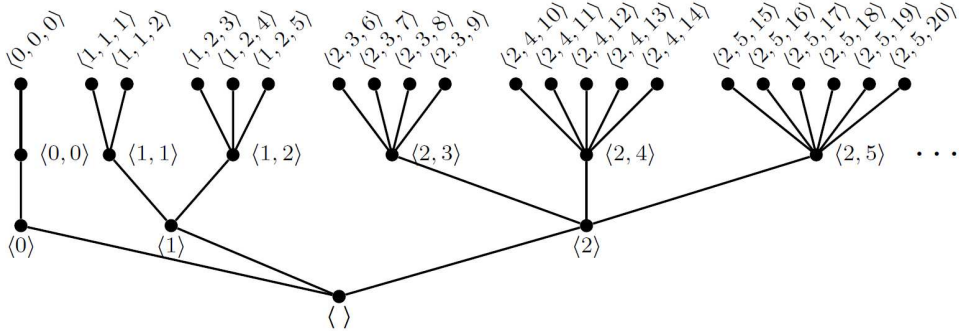}
\caption{Graph of $T_{2}$}
\label{graphT2}
\end{figure}

The next result is Theorem 3.23 of Dobrinen and Todorcevic in \cite{Ramsey-Class2} for $\alpha<\omega$.
\begin{theorem}[\cite{Ramsey-Class2}]
For each positive integer $n$, $(\mathcal{R}_{n}, \le, r)$  satisfies axioms $\bf{A.1}$-$\bf{A.4}$ and forms a topological Ramsey space.
\end{theorem}

\begin{definition}[$(\mathcal{R}^{\star}_{n}, \le , r)$] Assume $n$ is a positive integer and $T^{\star}_{1}, T^{\star}_{2},\dots, T^{\star}_{n}$ have been defined.  For each $i<\omega$, let
\begin{equation}
T_{n+1}^{\star}(i)=\left  \{ \left< \ \right>, \left < i \right >, \left< i\right >^{\frown}s: s\in T^{\star}_{n}(j)\ \& \ \frac{i(i+1)}{2}\le j <\frac{(i+1)(i+2)}{2} \right\}.
\end{equation}
Let $T_{n+1}^{\star}= \bigcup_{i<\omega} T_{n+1}^{\star}(i)$ and $(\mathcal{R}^{\star}_{n+1},\le,r)$ denote the triple $(\mathcal{R}(T^{\star}_{n+1}), \le, r)$.  
\end{definition}
The next lemma isolates the argument need to show by induction that for each positive integer $n$, $\mathcal{R}_{n}^{\star}$ satisfies $\bf{A.1}$-$\bf{A.4}$ and forms a topological Ramsey space.

\begin{lemma} \label{induction} For each positive integer $l$, if $\mathcal{R}^{\star}_{l}$ forms a topological Ramsey space then $\mathcal{R}^{\star}_{l+1}$ satisfies $\bf{A.4}$.
\end{lemma}
\begin{proof} Suppose that $l$ is a positive integer and $\mathcal{R}^{\star}_{l}$ forms a topological Ramsey space. By of Theorem 3.5 of Mijares in \cite{MijaresGalvin} applied to $\mathcal{R}^{\star}_{l}$, we find that for each pair of integers $k$ and $n$ with $k\le n$ there exists $m<\omega$ such that
\begin{equation}
\label{induction step}
r_{m}(T_{l}^{\star}) \rightarrow ( r_{(n+1)(n+2)/2}(T_{l}^{\star}))^{r_{(k+1)(k+2)/2}(T_{l}^{\star})}.
\end{equation}

  Next we prove a partition relation needed to show that $\bf{A.4}$ holds. Suppose that $\{ \Pi_{0}, \Pi_{1}\}$ is a partition of ${T_{l+1}^{\star}(m) \choose T_{l+1}^{\star}(k)}$. Let $m'$ be the unique integer such that $m'(m'+1) \le 2m < (m'+1)(m'+2)$.   Define a new partition $\{\overline{\Pi}_{0}, \overline{\Pi}_{1}\}$ of ${r_{m}(T^{\star}_{l})\choose r_{(k+1)(k+2)/2}(T^{\star}_{l})}$ by letting,
\begin{equation}
\overline{\Pi}_{0} = \left \{ S \in {r_{m}(T^{\star}_{l})\choose r_{(k+1)(k+2)/2}(T^{\star}_{l})} : cl( \{ \left< m' \right>^{\frown}s: s \in S \setminus r_{k(k+1)}(S) \}) \in \Pi_{0} \right \}
\end{equation}
and $\overline{\Pi}_{1}$ be its complement. By equation $(\ref{induction step})$ there exists $S \in{ r_{m}(T_{n}^{\star}) \choose r_{(n+1)(n+2)/2}(T_{l}^{\star})}$ and $j<2$ such that ${S \choose r_{(k+1)(k+2)/2}(T^{\star}_{l})}\subseteq \overline{\Pi}_{j}$.

If we let  $U = cl( \{ \left<m'\right>^{\frown} s : s \in  S \setminus r_{n(n+1)/2}(S)\})$ then $ U \in { T_{l+1}^{\star}(m) \choose T_{l+1}^{\star}(n)}$ and ${U \choose T_{l+1}^{\star}(k)}\subseteq \Pi_{j}$. Therefore for each pair of positive integers $k$ and $n$ with $k \le n$,
\begin{equation}
\label{finiteA.4}
T_{l+1}^{\star}(m) \rightarrow ( T_{l+1}^{\star}(n))^{T_{l+1}^{\star}(k)}.
\end{equation}

By the definition of $\mathcal{R}_{l+1}$, $\bf{A.4}$ for $\mathcal{R}_{l+1}$ is equivalent to the following: for all $k<\omega$, $T_{l+1}^{\star} \rightarrow (T^{\star}_{l+1})^{T_{l+1}^{\star}(k)}.$ To show that this partition relation holds, let $k$ be a positive integer and $\{\Pi_{0}, \Pi_{1}\}$ be a partition of ${T_{l+1}^{\star}\choose T_{l+1}^{\star}(k)}$. Since $n$ and $k$ where arbitrary in the proof of $(\ref{finiteA.4})$, there exists a strictly increasing sequence $(m_{n})_{n<\omega}$ such that for each $n<\omega$, $T^{\star}_{l+1}(m_{n}) \rightarrow (T^{\star}_{l+1}(n))^{T_{l+1}^{\star}(k)}$. Let $(S_{0}, S_{1}, \dots)$ be a sequence of trees such that for each $n<\omega$, $S_{n} \in { T_{l+1}^{\star}(m_{n}) \choose T_{l+1}^{\star}(i)}$ and ${S_{n} \choose T_{l+1}^{\star}(k)}$ is contained in one piece of the partition $\{\Pi_{0}, \Pi_{1}\}$. By the pigeonhole principle, there exists $j<2$ and a strictly increasing sequence $(n_{0}, n_{1}, \dots)$  such that for all $i<\omega$, $S_{n_{i}} \in { T_{l+1}^{\star}(m_{n_{i}})\choose T_{l+1}^{\star}(n_{i})}$ and ${S_{n_{i}} \choose T_{l+1}^{\star}(k)}\subseteq \Pi_{j}$. If $S = \bigcup_{i<\omega} S_{n_{i}}$ and $S'$ is any element of ${S \choose T_{l+1}^{\star}}$ then ${ S' \choose T_{l+1}^{\star}(k)} \subseteq \Pi_{j}$. Therefore for each positive integer $k$, 
$T_{l+1}^{\star} \rightarrow (T^{\star}_{l+1})^{T_{l+1}^{\star}(k)}$. holds. Equivalently, $\bf{A.4}$ holds for $\mathcal{R}_{l+1}$.
\end{proof}

\begin{theorem}\label{RamseySpaces} For each positive integer $n$, $(\mathcal{R}^{\star}_{n}, \le, r)$ satisfies $\bf{A.1}$-$\bf{A.4}$ and forms a topological Ramsey space.
\end{theorem}
\begin{proof}[Proof by induction on $n$] The base case when $n=1$ follows from Theorem $\ref{BaseCase}$.  Suppose that $\mathcal{R}^{\star}_{n}$ satisfies $\bf{A.1}$-$\bf{A.4}$ and forms a topological Ramsey space. By the abstract Ellentuck theorem it is enough to show that $(\mathcal{R}^{\star}_{n+1}, \le, r)$ satisfies $\bf{A.1}$-$\bf{A.4}$ and forms a closed subspace of $(\mathcal{AR}_{n+1})^{\omega}$. The proof that $\mathcal{R}^{\star}_{n+1}$ is a closed subspace of $(\mathcal{AR}_{n+1})^{\omega}$ and satisfies axioms $\bf{A.1}$-$\bf{A.3}$ follows by trivial modifications to the proofs of the same facts for the space $\mathcal{R}_{n+1}$ in \cite{Ramsey-Class2}. For this reason, we omit the proof that $\mathcal{R}^{\star}_{n+1}$ forms a closed subspace of $(\mathcal{AR}_{n+1})^{\omega}$ and satisfies axioms $\bf{A.1}$-$\bf{A.3}$.  The induction hypothesis and Lemma $\ref{induction}$ show that $\bf{A.4}$ holds for $\mathcal{R}_{n+1}^{\star}$.   By the abstract Ellentuck theorem, $\mathcal{R}^{\star}_{n+1}$ forms a topological Ramsey space.
\end{proof}

Next, for each positive integer $n$, we define maps $\gamma_{n}$ and $\Gamma_{n}$ that will be used to transfer an ultrafilter on $[T_{n}^{\star}]$ generated by a subset of $\mathcal{R}^{\star}_{n}$ to an ultrafilter on $[T_{n}]$ generated by a subset of $\mathcal{R}_{n}$.

\begin{definition}
Let $\{t_{0},t_{1},t_{2}, \dots\}$ and $\{s_{0},s_{1},s_{2}, \dots\}$ be the lexicographically increasing enumeration of $[T_{n}]$ and $[T_{n}^{\star}]$, respectively. Let $\gamma_{n}:[T_{n}^{\star}]\rightarrow [T_{n}]$ such that for all $i<\omega$,
\begin{equation}
 \gamma_{n}(s_{i})=t_{i}.
\end{equation}
 Let $\Gamma_{n}: \mathcal{R}_{n}^{\star}\rightarrow \mathcal{R}_{n}$ be the map given by 
\begin{equation}
\Gamma_{n}(S) = cl( \gamma_{n}''[S]).
\end{equation}
\end{definition}
\begin{remark} Let $n$ be a positive integer. $\gamma_{n}$ is bijective and $\Gamma_{n}$ is injective but not surjective. 
\end{remark}

\begin{theorem}
\label{ultrafilter} Let $n$ be a positive integer. $(\mathcal{R}_{n}^{\star}, \le^{*})$ is $\sigma$-closed, and if $\mathcal{G}$ is a generic filter for $(\mathcal{R}_{n}^{\star}, \le^{*})$ over some ground model $V$, then $\Gamma_{n}''\mathcal{G}$ generates an ultrafilter on $[T_{n}]$ that is selective for $\mathcal{R}_{n}$ but not Ramsey for $\mathcal{R}_{n}$ in $V[\mathcal{G}]$.
\end{theorem}
\begin{proof}
The proof that $(\mathcal{R}_{n}^{\star}, \le^{*})$ is $\sigma$-closed, and if $\mathcal{G}$ is a generic filter for $(\mathcal{R}_{n}^{\star}, \le^{*})$ over some ground model $V$, then $\Gamma_{n}''\mathcal{G}$ generates an ultrafilter on $[T_{n}]$ that is selective for $\mathcal{R}_{n}$ is completely analogous to the the proof for $\mathcal{R}_{1}$ given in the previous section. The only difference in the argument is an application of Theorem $\ref{RamseySpaces}$ instead of Theorem $\ref{RamseySpace}$. So we omit the proof and let $\mathcal{V}_{n}$ denote the selective for $\mathcal{R}_{n}$ ultrafilter generated by $\Gamma_{n}''\mathcal{G}$. 

 Next we construct a partition of ${ T_{n} \choose r_{2}(T_{n})}$ and show that the partition witnesses that $\mathcal{V}_{n}$ is not Ramsey for $\mathcal{R}_{n}$. For each $s \in{ T_{n} \choose r_{2}(T_{1})}$, we let $s'$ and $s''$ denote the two lexicographically smallest elements of $[s\setminus  r_{1}(s)]$. Notice that for each $s\in{ T_{n} \choose r_{2}(T_{n})}$ the length of the longest common initial segment of $\gamma^{-1}(s')$ and $\gamma^{-1}(s'')$ is either $n$ or $n-1$. For each $j<2$, let $\Pi_{j}$ denote the set of all $s \in{ T_{n} \choose r_{2}(T_{n})}$ such that length of the longest common initial segment of $\gamma^{-1}(s')$ and $\gamma^{-1}(s'')$ is $n-j$. For each $S\in \Gamma_{n}''\mathcal{G}$, ${ S\choose r_{2}(T_{n})}$ is neither a subset of $\Pi_{0}$ nor $\Pi_{1}$. Since $\mathcal{V}_{n}$ is generated by $\Gamma_{n}''\mathcal{G}$ it is not a Ramsey for $\mathcal{R}_{n}$ ultrafilter on $[T_{n}]$.
\end{proof}

\section{Selective but not Ramsey for $\bigotimes_{i=0}^{n} \mathcal{R}(S_{i})$} 
\label{section6}
In Sections $\ref{section4}$ and $\ref{section5}$ we only considered trees on $\omega$. In this section, in order to introduce the product of two spaces of the form $\mathcal{R}(T)$ and $\mathcal{R}(S)$ we must consider trees on $\omega^{2}$. Dobrinen, Mijares and Trujillo in \cite{GenRamsey-Class} have introduced a notion of product among special types of topological Ramsey spaces. Included among these special spaces are $\mathcal{R}(T_{i})$ and $\mathcal{R}^{\star}(T_{i})$ for $i<\omega$. In fact, for such spaces the product of $\mathcal{R}(S)$ and $\mathcal{R}(T)$ is defined by introducing the tree $S \otimes T$ on $\omega^{2}$ and letting $\mathcal{R}(S)\otimes\mathcal{R}(T)$ denote the triple $(\mathcal{R}(S \otimes T), \le r)$.

The simplest possible non-trivial product space is $\mathcal{R}_{1}\otimes\mathcal{R}_{1}$ which we denote by $\mathcal{H}^{2}$. It was first considered by Dobrinen, Mijares and Trujillo in \cite{GenRamsey-Class}.  The construction in \cite{GenRamsey-Class} was inspired by the work of Blass in \cite{BlassPpoints} which uses forcing to adjoin a p-point ultrafilter having two Rudin-Keisler incomparable predecessors and subsequent work of Dobrinen and Todorcebic in \cite{DobrinenTukey} which shows that the same forcing adjoins a p-point ultrafilter with two Tukey-incomparable p-point Tukey-predecessors. In \cite{GenRamsey-Class}, it is shown that forcing with $\mathcal{H}^{2}$ using almost-reduction adjoins an ultrafilter whose Rudin-Keisler predecessors form a four-element Boolean algebra. Before giving the general construction of the finite product we give the precise definition of the prototype example $\mathcal{H}^{2}$.

\begin{definition}[$(\mathcal{H}^{2}, \le, r)$]
Let
\begin{equation}
T_{1} \otimes T_{1} = \bigcup_{i<\omega} cl( \{ (s_{j}, t_{j})_{j<|s|} \in [\omega^{2}]^{<\omega} : s,t \in [T_{1}(i)]\}).
\end{equation}
We let $(\mathcal{H}^{2}, \le ,r)$ denote the space $(\mathcal{R}(T_{1}\otimes T_{1}), \le ,r)$.
\end{definition}
\begin{remark}
By the previous definition, for each $i<\omega$, $T_{1}\otimes T_{1}(i)= cl(\{ \left < (i,i), (j,k)\right>: j,k\le i\})$ and $T_{1}\otimes T_{1}= \bigcup T_{1}\otimes T_{1}(i)$. The elements of $\mathcal{H}^{2}$ are subtrees of $T_{1}\otimes T_{1}$ that are isomorphic to $T_{1}\otimes T_{1}$.
\end{remark}

The main theorem of this section implies that forcing with the similarly defined product $\mathcal{R}_{1}^{\star} \otimes\mathcal{R}_{1}^{\star}$ using almost-reduction adjoins an ultrafilter on $[T_{1}\otimes T_{1}]$ that is selective but not Ramsey for $\mathcal{H}^{2}$. In order to define the general finite product, we introduce a notion of finite product among the trees $\{T_{i}, T_{i}^{\star}, i<\omega\}$. If $S_{0} \otimes \cdots \otimes S_{k-1}$ is a finite product of $k$ such trees then $S_{0} \otimes \cdots \otimes S_{k-1}$ forms a tree on $\omega^{k}$. For example, $T_{1} \otimes T_{1}$ is a tree on $\omega^{2}$. 

\begin{definition}Suppose $k$ and $k'$ are positive integers and $s$ and $t$ are finite sequences of $k$-tuples and $k'$-tuples, respectively, such that $|s|=|t|$. We let $(s,t)$ denote the sequence on $\omega^{k+k'}$ givne by $(s_{i}, t_{i})_{i<\omega}$.  
\end{definition}
For example if $s=\left<1,2\right>$ and $t=\left <3,4 \right>$ then $(s,t)$ denotes the sequence $\left< (1,3), (2,4)\right>$ on $\omega^{2}$. Before giving the definition of the product of two triples we introduce the product of two trees.

\begin{definition}
Let $S$ and $T$ be trees on $\omega^{k}$ and $\omega^{k'}$, respectively. Assume that for all $s,t\in[S]\cup[T]$, $|s|=|t|$, $\pi_{0}''[S]=\{\left < (n,\dots,n)\right>\in \omega^{k} : n<\omega\}$ and $\pi_{0}''[T]=\{\left < (n,\dots,n)\right>\in \omega^{k'} : n<\omega\}$. We let
\begin{equation}
S \otimes T = \bigcup_{i<\omega} cl(\{ (s,t) \in (\omega^{(k+k')})^{<\omega} : s \in [S(i)] \ \& \ t \in[T(i)] \}).
\end{equation}
\end{definition}

\begin{remark} For positive integer $n$ and each $m> n$, there exists $s,t\in[S]\cup[T]$ such that $|s|\not=|t|$. Therefore the product $T_{n} \otimes T_{m}$ is not well-defined. For each $m> n$, the elements of $[T_{n}]$ can be extended to length $m$ sequences by repeating the last element of a given sequence for the final $(m-n)$-elements of the extended sequence. (For example, the length 2 sequence $(2,3)$ would extend to the length 4 sequence $(2,3,3,3).$) The space constructed using the extended tree is isomorphic to the original space. In this way the product $T_{n} \otimes T_{m}$ is well-defined.
\end{remark}

\begin{definition}
Suppose that $\left < S_{i} : i\le n \right>$ is a finite sequence of trees where each $S_{i}$ is one of the trees $T_{j}$ for some $j<\omega$. Without loss of generality we may extend all of the trees so that for each $s,t\in \bigcup_{i\le n} [S_{i}]$, $|s|=|t|$. If $n=1$ then we let 
$
\bigotimes_{i=0}^{1}S_{i} = S_{0} \otimes S_{1}.
$ If $n>1$ then we recursively define the product by letting,
$
 \bigotimes_{i=0}^{n} S_{i} = S_{n} \otimes \bigotimes_{i=0}^{n-1} S_{i}.
$ For each sequence let,
$
\bigotimes_{i=0}^{n} \mathcal{R}(S_{i}) = \mathcal{R}( \bigotimes_{i=0}^{n} S_{i})
$
and
$
\bigotimes_{i=0}^{n} \mathcal{R}^{\star}(S_{i}) = \mathcal{R}( \bigotimes_{i=0}^{n} S^{\star}_{i}).
$
\end{definition}

The next theorem follows from a more general theorem in \cite{GenRamsey-Class} about products of sequences of structures in a relational language. 
\begin{theorem}[\cite{GenRamsey-Class}] \label{ProductRamseySpace} If $\left < S_{i} : i\le n \right>$ is a finite sequence of trees where each $S_{i}$ is one of the trees $T_{j}$ or $T_{j}^{\star}$ for some $j<\omega$, then $(\bigotimes_{i=0}^{n} \mathcal{R}(S_{i}) , \le , r)$ satisfies $\bf{A.1}$-$\bf{A.4}$ and forms a topological Ramsey space.
\end{theorem}

\begin{definition} Suppose that $\left < S_{i} : i\le n \right>$ is a finite sequence of trees where each $S_{i}$ is one of the trees $T_{j}$ for some $j<\omega$. Let $\{t_{0},t_{1},t_{2}, \dots\}$ and $\{s_{0},s_{1},s_{2}, \dots\}$ be a lexicographically non-decreasing enumeration of $ [\bigotimes_{i=0}^{n} S_{i}]$ and $[\bigotimes_{i=0}^{n} S^{\star}_{i}]$, respectively. Let $\boldsymbol{\gamma}:[\bigotimes_{i=0}^{n} S_{i}]\rightarrow [\bigotimes_{i=0}^{n} S^{\star}_{i}]$ such that for all $i<\omega$,
\begin{equation}
 \boldsymbol{\gamma}(s_{i})=t_{i}.
\end{equation}
 Let $\boldsymbol{\Gamma}: \bigotimes_{i=0}^{n} \mathcal{R}^{\star}(S_{i})\rightarrow \bigotimes_{i=0}^{n} \mathcal{R}(S_{i})$ be the map given by 
\begin{equation}
\boldsymbol{\Gamma}(S) = cl( \boldsymbol{\gamma}''[S]).
\end{equation}
\end{definition}
\begin{remark} For each $\left < S_{i} : i\le n \right>$  sequence, $\boldsymbol{\gamma}$ is bijective and $\boldsymbol{\Gamma}$ is injective but not surjective.
\end{remark}

\begin{theorem}
\label{ultrafilterProd} Suppose that $\left < S_{i} : i\le n \right>$ is a finite sequence of trees where each $S_{i}$ is one of the trees $T_{j}$ for some $j<\omega$. $(\bigotimes_{i=0}^{n} \mathcal{R}^{\star}(S_{i}), \le^{*})$ is $\sigma$-closed, and if $\mathcal{G}$ is a generic filter for $(\bigotimes_{i=0}^{n} \mathcal{R}^{\star}(S_{i}), \le^{*})$ over some ground model $V$, then $\boldsymbol{\Gamma}''\mathcal{G}$ generates an ultrafilter on $[\bigotimes_{i=0}^{n} S_{i}]$ that is selective for $\bigotimes_{i=0}^{n} \mathcal{R}(S_{i})$ but not Ramsey for $\bigotimes_{i=0}^{n} \mathcal{R}(S_{i})$ in $V[\mathcal{G}]$.
\end{theorem}
\begin{proof}
The proof that $(\bigotimes_{i=0}^{n} \mathcal{R}^{\star}(S_{i}), \le^{*})$ is $\sigma$-closed, and if $\mathcal{G}$ is a generic filter for $(\bigotimes_{i=0}^{n} \mathcal{R}^{\star}(S_{i}), \le^{*})$ over some ground model $V$, then $\boldsymbol{\Gamma}''\mathcal{G}$ generates an ultrafilter on $[\bigotimes_{i=0}^{n} S_{i}]$ that is selective for $\bigotimes_{i=0}^{n} \mathcal{R}(S_{i})$ is completely analogous to the the proof for $\mathcal{R}_{1}$ given in the  Section $\ref{section4}$. The only difference in the argument is an application of Theorem $\ref{ProductRamseySpace}$ instead of Theorem $\ref{RamseySpace}$. So we omit the proof and let $\boldsymbol{\mathcal{V}}$ denote the selective for $\bigotimes_{i=0}^{n} \mathcal{R}^{\star}(S_{i})$ ultrafilter generated by $\boldsymbol{\Gamma}''\mathcal{G}$.

Let $\tau_{0}$ be the map which takes a tree on $\omega^{n+1}$ to a tree on $\omega$ by sending each sequence of $(n+1)$-tuples to the sequence of first elements of the $(n+1)$-tuple. (For example, if $n=2$ and $S= cl(\{\left< (1,3), (2,4)\right>\})$ then $\tau_{0}(S)=cl(\{\left< 1,2\right>\})$.

 Next we construct a partition of ${ \bigotimes_{i=0}^{n} S_{i}, \choose r_{2}(\bigotimes_{i=0}^{n} S_{i})}$ and show that the partition witnesses that $\boldsymbol{\mathcal{V}}$ is not Ramsey for $\bigotimes_{i=0}^{n} \mathcal{R}^{\star}(S_{i})$. For each $s \in{ \bigotimes_{i=0}^{n} S_{i}, \choose r_{2}(\bigotimes_{i=0}^{n} S_{i})}$, we let $s'$ and $s''$ denote any two elements of $[s\setminus  r_{1}(s)]$ such that $\tau_{0}(s'), \tau_{0}(s'')$ are the two lexicographically smallest elements of $[\tau_{0}(s)\setminus  r_{1}(\tau_{0}(s))]$. Notice that for each $s \in{ \bigotimes_{i=0}^{n} S_{i}, \choose r_{2}(\bigotimes_{i=0}^{n} S_{i})}$ the length of the longest common initial segment of $\boldsymbol{\gamma}^{-1}(\tau_{0}(s'))$ and $\boldsymbol{\gamma}^{-1}(\tau_{0}(s''))$ is either $j$ or $j-1$ where $j$ is that natural number such that $S_{0}=T_{j}$. For each $k<2$, let $\Pi_{k}$ denote the set of all $s \in{ \bigotimes_{i=0}^{n} S_{i}, \choose r_{2}(\bigotimes_{i=0}^{n} S_{i})}$ such that length of the longest common initial segment of $\boldsymbol{\gamma}^{-1}(\tau_{0}(s'))$ and $\boldsymbol{\gamma}^{-1}(\tau_{0}(s''))$ is $j-k$. For each $S\in \boldsymbol{\Gamma}''\mathcal{G}$, ${ S \choose r_{2}(\bigotimes_{i=0}^{n} S_{i})}$ is neither a subset of $\Pi_{0}$ nor $\Pi_{1}$. Since $\boldsymbol{\mathcal{V}}$ is generated by $\boldsymbol{\Gamma}''\mathcal{G}$ it is not a Ramsey for $\bigotimes_{i=0}^{n} \mathcal{R}(S_{i})$ ultrafilter on $ [\bigotimes_{i=0}^{n} S_{i}]$. 
\end{proof}
\section{Conclusion}
\label{section7}
 For each positive integer $k$, we have presented countably many examples of trees on $\omega^{k}$ where forcing can be used to adjoin an ultrafilter on $[T]$ that is selective but not Ramsey for $\mathcal{R}(T)$. In each case, a new topological Ramsey space $\mathcal{R}^{\star}(T)$, a map $\Gamma: \mathcal{R}^{\star}(T) \rightarrow \mathcal{R}(T)$ and a partition $\{\Pi_{0}, \Pi_{1}\}$ of ${T\choose r_{2}(T)}$ where constructed in such a way that for all $S\in \Gamma''\mathcal{R}^{\star}(T)$, neither ${S\choose r_{2}(T)}\subseteq \Pi_{0}$ nor ${S\choose r_{2}(T)}\subseteq \Pi_{1}$. The main results follow by showing that if $\mathcal{G}$ is generic for $(\mathcal{R}^{\star}(T), \le^{*})$ then $\Gamma''\mathcal{G}$ generates a selective but not Ramsey for $\mathcal{R}(T)$ ultrafilter on $[T]$. In each case, the partition $\{\Pi_{0}, \Pi_{1}\}$ witnesses that the generated ultrafilter cannot be Ramsey for $\mathcal{R}(T)$.

 Dobrinen and Todorcevic in \cite{Ramsey-Class2} have also introduced the spaces $\mathcal{R}_{\alpha}$ where $\omega\le \alpha < \omega_{1}$. These spaces are constructed  from trees of infinite height in a slightly different manner than those considered in this article. In these cases it is possible to construct trees $T_{\alpha}$ and $T^{\star}_{\alpha}$, and a modified version of $\mathcal{R}^{\star}_{\alpha}$ for $\omega\le \alpha < \omega_{1}$. However, the partition given in the finite case can not be extended to the case for $\omega\le \alpha < \omega_{1}$ since the trees being used have infinite height. In particular the next question remains open.
\begin{question}
For $\alpha$ between $\omega$ and $\omega_{1}$, are the notions of selective for $\mathcal{R}_{\alpha}$ and Ramsey for $\mathcal{R}_{\alpha}$ equivalent?
\end{question}

For each positive integer $n$, let $\mathcal{H}^{n}$ denote the space $\otimes_{i=1}^{n} \mathcal{R}_{1}$. By Theorem $\ref{ultrafilterProd}$ forcing with $\otimes_{i=1}^{n} \mathcal{R}_{1}^{\star}$ using almost-reduction adjoins an ultrafilter that is selective but not Ramsey for $\mathcal{H}^{n}$. Dobrinen, Mijares and Trujillo in \cite{GenRamsey-Class} have also defined the topological Ramsey spaces $\mathcal{H}^{\alpha}$ for $\omega\le \alpha< \omega_{1}$. For similar reason to the $\mathcal{R}_{\alpha}$ case our methods fail to produce an ultrafilter that is selective but not Ramsey for $\mathcal{H}^{\alpha}$. Hence our next question also remains open.
\begin{question}
For $\alpha$ between $\omega$ and $\omega_{1}$, are the notions of selective for $\mathcal{H}^{\alpha}$ and Ramsey for $\mathcal{H}^{\alpha}$ equivalent?
\end{question}

All the spaces studied in this paper except the Ellentuck space either support ultrafilters which are selective but not Ramsey for the space or it is unknown if the notions of selective and Ramsey for the space are equivalent.  In fact the following question is still open.
\begin{question}
Is the Ellentuck space the only topological Ramsey space for which the notions of selective and Ramsey for the space are equivalent?
\end{question}

On a final note, we have also studied which properties of a tree $T$ on $\omega$ lead to triples $(\mathcal{R}(T),\le, r)$ that satisfy axioms $\bf{A.1}$-$\bf{A.4}$ and form a topological Ramsey space. In the case of finite trees on $\omega$, with a tedious proof, it is possible to show that for each such tree $T$ there is an $i<\omega$ such that $(\mathcal{R}(T),\le^{*})$ and $(\mathcal{R}(T_{i}), \le^{*})$ are densely bi-embeddable in one another. For example for each positive integer $n$, $(\mathcal{R}^{\star}(T_{n}),\le^{*})$ and $(\mathcal{R}(T_{n+1}), \le^{*})$ are densely bi-embeddable in one another.  Although these spaces are not necessarily identical they are similar enough that the methods of the paper work in this more general setting. We have omitted these proofs and results form this work in order to make the proofs easier to understand.

\bibliographystyle{model1-num-names}
\bibliography{Trujillo}

\end{document}